\documentclass[12pt, psamsfonts]{amsart}
\usepackage{amssymb,amsfonts,amsthm,amsmath,epsfig,hhline}
\usepackage{verbatim, enumerate}    

\usepackage[arc,all,graph,frame]{xy}

\theoremstyle{plain}
\newtheorem{theorem}{Theorem}

\newtheorem{lemma}{Lemma}

\theoremstyle{definition}

\newtheorem*{remark}{Remark}

\def\d{\underbar{\em d}}

\title{A sharp refinement of a result of  {Z}verovich--{Z}verovich} 

\author{Grant Cairns}
\author{Stacey Mendan}
\author{Yuri Nikolayevsky}

\address{Dept of Mathematics and Statistics, La Trobe University, Melbourne, Australia 3086}
\email{G.Cairns@latrobe.edu.au}
\email{spmendan@students.latrobe.edu.au}
\email{Y.Nikolayevsky@latrobe.edu.au}

\begin{document}

\maketitle

\begin{abstract}
For a finite sequence of positive integers to be the degree sequence of a finite graph, Zverovich and Zverovich gave a sufficient condition  involving  only the length of the sequence, its maximal element and its minimal element. In this paper we give a sharp refinement of   {Z}verovich--{Z}verovich's result.
\end{abstract}

\section{Introduction}
A finite sequence $\d=(d_1,\dots,d_n)$ of positive  integers is \emph{graphic} if it occurs as the sequence of vertex degrees of a simple graph. The classic theorem of  Erd\H{o}s and Gallai Theorem gives a necessary and sufficient condition for a sequence to be graphic (see \cite{TV,TVW,YL}). A theorem of Zverovich and Zverovich gives a sufficient condition involving  only the length of the sequence, its maximal element and its minimal element. Their result can be stated in the following equivalent form.

\begin{theorem}[{\cite[Theorem 6]{ZZ}}]\label{T:zz}
Suppose that $\emph{\d}$ is a decreasing sequence  of positive integers with even sum. Let $a$ (resp.~$b$) denote the maximal (resp.~minimal) element of $\emph{\d}$. Then  $\emph{\d}$  is graphic if 
\begin{equation}\label{E:zz}
nb\geq  \frac{(a  +  b +1)^2}{4  }.
\end{equation}
\end{theorem}

It is known that this result is not sharp (see  \cite{BHJW}). A sharp bound in the case $b =1$ was given in \cite{CSZZ}. The main aim of this paper is to prove the following result, which is sharp for all $a ,b $ and $n$. 

\begin{theorem}\label{T:zz2}
Suppose that $\emph{\d}$ is a decreasing sequence  of positive integers with even sum. Let $a$ (resp.~$b$) denote the maximal (resp.~minimal) element of $\emph{\d}$. Then  $\emph{\d}$  is graphic if 
\begin{equation}\label{E:zz2}
nb\geq  \begin{cases}
\left\lfloor\dfrac{(a  +  b +1)^2}{4    }\right\rfloor-1&: \ \text{if } b\ \text{is odd, or }a +b \equiv 1\pmod 4,\\ \\
\left\lfloor\dfrac{(a  +  b +1)^2}{4    }\right\rfloor&: \ \text{otherwise},
\end{cases}
\end{equation}
where $\lfloor.\rfloor$ denotes the integer part. Moreover, for any triple $(a,b,n)$ of positive integers with $b< a<n$ that fails \eqref{E:zz2}, there is a nongraphic sequence of length $n$ having even sum with maximal element $a$ and  minimal element $b$.
\end{theorem}

The paper is organised as follows. In Section \ref{S:cond} we examine condition \eqref{E:zz2} and rewrite it in a more convenient form. We then prove that condition \eqref{E:zz2} is sufficient in Section \ref{S:suff}. The sharpness is shown in Section \ref{S:nec}. To establish this we first prove the following result in Section \ref{S:2e}, which may be of independent interest.
Here, and in sequences throughout this paper, the superscripts indicate the number of repetitions of the entry. 

\begin{theorem}\label{T:ab}
Consider natural numbers $b<a<n$ and suppose that $as+b(n-s)$ is even. Then for $0<s<n$, the sequence $(a^s,b^{n-s})$ is graphic if and only if $s^2-(1+a+b)s+nb\geq 0$.
 \end{theorem}

\begin{remark}
The assumption $b<a<n$ is not restrictive. All sequences with $a\geq n$ are obviously nongraphic. For $a=b$, it follows from Theorem \ref{T:zz2} that  $(a^n)$ is graphic if and only if $an$ is even and $a<n$.\end{remark}

Throughout the following,   $\d=(d_1,\dots,d_n)$ denotes a decreasing sequence  with maximal element $d_1=a$ and  minimal element $d_n=b$. 

\section{The hypothesis}\label{S:cond}

We claim that  the inequality \eqref{E:zz2} can be conveniently expressed according to the following four disjoint, exhaustive cases:
\begin{enumerate}[{\rm (a)}]
  \item[(I)] \label{it:a}
  If $a+b+1\equiv 2bn \pmod 4$, then
  $(a+b+1)^2 \le 4bn$.

  \item[(II)] \label{it:b}
  If $a+b+1\equiv 2bn+2 \pmod 4$, then $(a+b+1)^2 \le 4bn+4$.

  \item[(III)] \label{it:c}
If $a+b$ is even and $bn$ is even, then
 $ (a+b+1)^2 \le 4bn+1$.
 
  \item[(IV)] \label{it:d}
  If $n,a,b$ are all odd, then $ (1+a+b)^2 \le 4bn+5$.
\end{enumerate}

First note that  in cases (I) and (II), we have $\left\lfloor\frac{(a  +  b +1)^2}{4    }\right\rfloor= \frac{(a  +  b +1)^2}{4    }$, while in cases (III) and (IV), $\left\lfloor\frac{(a  +  b +1)^2}{4    }\right\rfloor= \frac{(a  +  b +1)^2}{4    }-\frac14$. 

Consider case (I). There are two subcases to consider here. First,  if $b$ is even, then $a+b\equiv -1\pmod4$, so  \eqref{E:zz2} reads
\[
nb\ge \left\lfloor\frac{(a  +  b +1)^2}{4    }\right\rfloor=\frac{(a  +  b +1)^2}{4    }\]
or equivalently $(a+b+1)^2 \le 4bn$, as required.
The other subcase of case (I) is where $b$ is odd. Here,  \eqref{E:zz2} reads
\[
nb\ge \left\lfloor\frac{(a  +  b +1)^2}{4    }\right\rfloor-1=\frac{(a  +  b +1)^2-4}{4    }\]
or equivalently $(a+b+1)^2 \le 4bn+4$. Note that both sides of this inequality are multiples of 4. We claim that equality is impossible here, and so the condition is equivalent to  $(a+b+1)^2 \le 4bn$. Indeed, if $(a+b+1)^2 = 4bn+4$, then as $a+b+1\equiv 2bn \pmod 4$, we would have $4b^2n^2\equiv 4bn+4\pmod 8$. Hence 
$b^2n^2\equiv bn+1\pmod 2$. But this is impossible, as $x^2\equiv x\pmod 2$ for all $x$.

Consider case (II). Here, if $b$ is even, $a+b\equiv 1\pmod4$. So, regardless of whether $b$ is even or odd,  \eqref{E:zz2} reads
\[
nb\ge \left\lfloor\frac{(a  +  b +1)^2}{4    }\right\rfloor-1=\frac{(a  +  b +1)^2-4}{4    }\]
or equivalently $(a+b+1)^2 \le 4bn+4$, as required.

Consider case (III).  There are two subcases to consider here. First,  if $b$ is even, then $a$ is necessarily even, so $a+b\not\equiv 1\pmod4$, and hence  \eqref{E:zz2} reads
\[
nb\ge \left\lfloor\frac{(a  +  b +1)^2}{4    }\right\rfloor=\frac{(a  +  b +1)^2}{4    }-\frac14\]
or equivalently $(a+b+1)^2 \le 4bn+1$, as required. 
The other subcase of case (III) is where $b$ is odd. Here,  \eqref{E:zz2} reads
\[
nb\ge \left\lfloor\frac{(a  +  b +1)^2}{4    }\right\rfloor-1=\frac{(a  +  b +1)^2}{4    }-\frac14-1\]
or equivalently $(a+b+1)^2 \le 4bn+5$. We claim that equality is not possible and thus, as $(a+b+1)^2\equiv 1\pmod 4$, the condition is equivalent to  $(a+b+1)^2 \le 4bn+1$. Indeed, since $bn$ is even, 
$4bn+5\equiv 5 \pmod 8$. But  $a+b$ is even, say $a+b=2k$, so we have
$(a+b+1)^2=4(k^2+k)+1\equiv 1\pmod 8$, as $k^2+k$ is even for all $k$.
So $(a+b+1)^2 = 4bn+5$ is impossible.

In case (IV), $b$ is odd, so  \eqref{E:zz2} reads
\[
nb\ge \left\lfloor\frac{(a  +  b +1)^2}{4    }\right\rfloor-1=\frac{(a  +  b +1)^2}{4    }-\frac14-1\]
or equivalently $(a+b+1)^2 \le 4bn+5$, as claimed.

\section{Proof of sufficiency}\label{S:suff}
Suppose that $\d$ has even sum and that inequality \eqref{E:zz2} holds. 
We consider the 4 cases (I) -- (IV) given in Section \ref{S:cond}.
 The sufficiency in case (I) follows immediately from Theorem \ref{T:zz}. For the other cases, we employ similar ideas to those of \cite{ZZ}.
Let us first  recall some terminology and results of \cite{ZZ}. A number $k$ is called a \emph{strong index}, if $d_k \ge k$. Note that the set of strong indices is nonempty as $d_1\ge 1$, but not all indices are strong as $d_n=b<n$. The maximal strong index is denoted $k_m$. For $j \ge 0, \; n_j:=\#\{i: d_i=j\}$. By \cite[Theorem~3]{ZZ}, a sequence $\underline{d}$ is graphic if and only if it has an even sum and $r_k \le k(n-1)$ for all strong indices $k$, where $r_k=\sum_{i=1}^k(d_i+in_{k-i})$.

Let $k$ be a strong index of $\underline{d}$. If $k \le b$, then $n_{k-i}=0$ for all $1 \le i \le k$, so $r_k=\sum_{i=1}^k d_i \le ka \le k(n-1)$. So we may suppose that $k > b$. 

\begin{lemma}\label{L:ineq}
 We have
\begin{equation}\label{eq:rkkm}
    r_k \le k(n-1)+k_m(a+b+1)-k_m^2-bn,
\end{equation}
with equality only possible when $k=k_m$ and $\emph{\d}$ has the form $\emph{\d}=(a^{k_m} b^{n-k_m})$.
\end{lemma}

\begin{proof}
We have $\sum_{i=1}^k d_i \le ka$, which becomes an equality only if $d_1=\dots=d_k=a$. Moreover, $\sum_{i=1}^k in_{k-i}=\sum_{j=0}^{k-1} (k-j)n_j \le (k-b)\sum_{j=0}^{k-1} n_j$, with equality only possible when all the $n_0, n_1, \dots, n_{k-1}$, but $n_b$ are zeros; that is, when $\big(d_i \le k-1 \Rightarrow d_i = b\big)$. Furthermore, \[
\sum_{j=0}^{k-1} n_j=\#\{i\in\{1,2,\dots,n\}: d_i \le k-1\} \le \#\{i: n\ge i > k_m\},
\]
 since for $i \le k_m$ we have $d_i \ge d_{k_m} \ge k_m \ge k$. Thus $\sum_{j=0}^{k-1} n_j \le n-k_m$, which becomes an equality only when $d_{k_m+1} \le k-1$. Thus $\sum_{i=1}^k in_{k-i} \le (k-b)(n-k_m)$, with equality only possible when $d_{k_m+1} = \dots = d_n=b$; indeed, if $d_{k_m+1} \le k-1$ and $\big(d_i \le k-1 \Rightarrow d_i = b\big)$, then $d_{k_m+1}=b$ and so $d_{k_m+1} = \dots = d_n=b$.
Thus \begin{equation}\label{eq:rk}
    r_k \le ka+(k-b)(n-k_m),
\end{equation}
with  equality only possible when $\underline{d}=(a^k,d_{k+1}, \dots,d_{k_m}, b^{n-k_m})$. 
As $a \ge d_{k_m} \ge k_m$, we have $a+1-k_m \ge 1$. Thus,  using $k\leq k_m$, inequality \eqref{eq:rk} gives
\begin{align*}
r_k \le ka+(k-b)(n-k_m)&=k(n-1)+k(a+1-k_m)+bk_m-bn \\
&\le k(n-1)+k_m(a+1-k_m)+bk_m-bn\\
&= k(n-1)+k_m(a+b+1)-k_m^2-bn,
\end{align*}
as required, with equality only possible when $k=k_m$ and $\d=(a^{k_m}, b^{n-k_m})$.
\end{proof}

Now consider cases (II) -- (IV) separately. In case~(III), the maximal value of the right-hand side of \eqref{eq:rkkm}, regarded as a quadratic in $k_m$, is attained at $k_m=\frac12(a+b+1 \pm 1)$ and is equal to $k(n-1)+\frac14 ((a+b+1)^2-1)-bn$. Thus the inequality $(a+b+1)^2 \le 4bn+1$ implies $r_k \le k(n-1)$ and so the sequence $\underline{d}$ is graphic by \cite[Theorem~3]{ZZ}.

In case~(II), the (unique) maximal value of the right-hand side of \eqref{eq:rkkm} for an integer $k_m$ is attained at $k_m=\frac12(a+b+1)$ and is equal to $k(n-1)+\frac14 (a+b+1)^2-bn$, but the sum of the sequence $(a^{\frac12(a+b+1)} ,b^{n-\frac12(a+b+1)})$ is odd, so the inequality \eqref{eq:rkkm} becomes strict, hence $r_k \le k(n-1)+\frac14 (a+b+1)^2-bn-1$. Thus the inequality $(a+b+1)^2 \le 4bn+4$ implies $r_k \le k(n-1)$ and therefore the sequence $\underline{d}$ is graphic by \cite[Theorem~3]{ZZ}.

In case (IV), $\d$ is not of the form $(a^s,b^{n-s})$ since, as $a,b$ and $n$ are all odd, the sequences of the form $(a^s,b^{n-s})$ have odd sum, contrary to our hypothesis. So by Lemma \ref{L:ineq}, 
the inequality~\eqref{eq:rkkm} is not strict.
As the maximum of the right hand side of \eqref{eq:rkkm} is attained for $k_m=\frac12(a+b+1 \pm 1)$, we get $r_k \le k(n-1)+\frac14 ((a+b+1)^2-1)-bn-1$, which implies  $r_k \le k(n-1)$ whenever $(a+b+1)^2 \le 4bn+5$.

\section{Two-element sequences}\label{S:2e}

\begin{proof}[Proof of Theorem \ref{T:ab}]  We apply the Erd\H{o}s--Gallai Theorem, which says that  $\d$ is graphic if and only if its sum is even and for each integer $k$ with $1 \leq  k \leq  n$,
\begin{equation*}
\sum_{i=1}^k d_i\leq  k(k-1)+ \sum_{i=k+1}^n\min\{k,d_i\}.\tag{EG}
\end{equation*}
For the sequence $\d=(a^s,b^{n-s})$, we consider (EG)  in 5 cases:
\begin{enumerate}
\item[(i)] If $k>s$ and $k\leq b$, then (EG) reads
\[
as+b(k-s)\leq k(k-1)+(n-k)k=k(n-1).\]
We have $as+b(k-s)=bk+s(a-b)<bk+k(a-b)=ka\leq k(n-1)$, so (EG) holds in this case.

\item[(ii)] If $k\leq s$ and $k\leq b$, then (EG) reads
\[
ak\leq k(k-1)+(n-k)k=k(n-1),\]
which is true as $a\leq n-1$.

\item[(iii)] If $k\leq s$ and $a<k$, then (EG) reads
\[
ak\leq k(k-1)+(s-k)a+(n-s)b.
\]
As $a\leq k-1$, we have $ak\leq k(k-1)\leq k(k-1)+(s-k)a+(n-s)b$, so (EG) holds in this case.

\item[(iv)] If $k>s$ and $k>b$, then (EG) reads
$as+b(k-s)\leq k(k-1)+(n-k)b$; that is
\begin{equation}\label{E:delt}
k^2-k(1+2b)+nb+bs-as\geq 0.
\end{equation}

\item[(v)] If $k\leq s$ and $b<k\leq a$, then (EG) reads
\[
ak\leq k(k-1)+(s-k)k+(n-s)b=(s-1)k+(n-s)b.
\]
This condition holds if $a\leq s-1$. If $a\geq s$, (EG) is $(a-s+1)k\leq (n-s)b$ and the most restrictive case occurs when $k=s$. Here the condition is 
\begin{equation}\label{E:delt0}
s^2-(1+a+b)s+nb\geq 0.
\end{equation}
\end{enumerate}
From the above we see that $(a^s,b^{n-s})$ is graphic if and only if \eqref{E:delt0} holds and \eqref{E:delt} holds for all $k>s$, $k>b$. 

\begin{lemma}\label{L:delta}
$s^2-(1+a+b)s+nb\geq 0$ if and only if $k^2-k(1+2b)+nb+bs-as\geq 0$ for all $k\in\{s,s+1,\dots,n\}$. 
\end{lemma}

\begin{proof}
Fix $n,a,b,s$ and let  $\Delta_k=k^2-k(1+2b)+nb+sb-as$. So $\Delta_s=s^2-(1+a+b)s+nb$, and hence one direction in this lemma is trivial. For the other direction, note that $\Delta_k$ is quadratic in the integer $k$ and takes its minimum value at the integers $b$ and $b+1$. The minimum value of $\Delta_k$ is 
\[
\Delta_b=b^2-b(1+2b)+nb+bs-as=bn-b^2-b-s(a-b).
\]
Suppose that $\Delta_s\geq 0$ and that $\Delta_k< 0$ for some integer $k>s$. Then $s\leq b-1$ and $\Delta_b<0$, so $s(a-b)>bn-b^2-b$ and hence $(b-1)(a-b) \geq s(a-b)>bn-b^2-b$. Expanding gives
 so
$ab-a+2b >bn$. Hence, as $b< a$,
\[
bn<ab-a+2b< ab+b=b(a+1)
\]
and so $n<a+1$. But this is impossible as $a< n$, by assumption.
\end{proof}
This completes the proof of Theorem \ref{T:ab}.\end{proof}

\section{Proof of Necessity}\label{S:nec}

Assume that \eqref{E:zz2} fails 
for the triple $(a,b,n)$, where $b< a<n$. We will exhibit a nongraphic sequence $\d$ of length $n$ having even sum with maximal element $a$ and  minimal element $b$. We consider the same four cases (I) -- (IV) given in Section \ref{S:cond}.
So our assumption is respectively:
\begin{enumerate}[{\rm (a)}]
  \item[(I)] \label{it:a}
   $a+b+1\equiv 2bn \pmod 4$, and
  $(a+b+1)^2 > 4bn$.

  \item[(II)] \label{it:b}
   $a+b+1\equiv 2bn+2 \pmod 4$,  and
 $(a+b+1)^2 > 4bn+4$.

  \item[(III)] \label{it:c}
$a+b$ is even and $bn$ is even,  and
 $ (a+b+1)^2 > 4bn+1$.
 
  \item[(IV)] \label{it:d}
   $n,a,b$ are all odd,  and
 $ (1+a+b)^2 > 4bn+5$.
\end{enumerate}

In cases (I) -- (III), the proposed sequences have the form $\d=(a^s,b^{n-s})$, where respectively:
\[
\text{(I)}\quad  s={\frac{a+b+1}{2}};\qquad\text{(II)}\quad   s={\frac{a+b+3}{2}};\qquad\text{(III)}\quad   s={\frac{a+b}{2}}.
\]
In case (I), $s^2-(1+a+b)s+nb=-\frac{(1+a+b)^2}4+nb < 0$ and so $\d$ is nongraphic by Theorem~\ref{T:ab}. Moreover, $\d$ has sum 
\[
as+b(n-s)=\frac{a(a+b+1)}{2}+\frac{b(2n-(a+b+1))}{2}=\frac{(a-b)(a+b+1)+2bn}{2}.
\]
In case (I), $a+b$ is odd, so $a-b$ is odd and $a+b+1\equiv 2bn \pmod 4$, so 
\[
(a-b)(a+b+1)+2bn\equiv  2bn+2bn\equiv0\pmod 4.
\]
Thus $\d$ has even sum. Cases (II) and (III) are treated in exactly the same manner.

In case (IV), $a,b,n$ are all odd. Here we consider the decreasing sequence
$\d=(d_1,\dots,d_n)=(a^{\frac{a+b}{2}},b+1,b^{\frac{2n-(a+b)-2}{2}})$.
Let $s=\frac{a+b}{2}$. We will show  that the sequence fails the $s$-th inequality of the Erd\H{o}s-Gallai Theorem. By assumption,
$(2s+1)^2>4nb+5$, so 
 $nb<s^2+s-1$. As $nb$ is odd, this implies $nb\le  s^2+s-3$. Thus $b+1\le  s^2+s-nb+b-2$. Therefore,
\begin{align*}
\sum_{i=1}^s d_i=as & >as-2=(2s-b)s-2\\
& =s(s-1)+[s^2+s-nb+b-2]+(n-s-1)b\\
& \geqslant s(s-1)+b+1+(n-s-1)b\\
& = s(s-1)+\sum_{i=s+1}^n \min\{s,d_i\}.
\end{align*}
So (EG) fails for $k=s$. Finally, $\d$ has even sum  since, as $a,b,n \equiv 1 \pmod 2$,
\[
as+(b+1)+(n-s-1)b \equiv s+0+s \equiv 0 \pmod 2.
\]
This completes the proof of Theorem \ref{T:zz2}.

\providecommand{\bysame}{\leavevmode\hbox to3em{\hrulefill}\thinspace}
\providecommand{\MR}{\relax\ifhmode\unskip\space\fi MR }
\providecommand{\MRhref}[2]{%
  \href{http://www.ams.org/mathscinet-getitem?mr=#1}{#2}
}
\providecommand{\href}[2]{#2}


\begin{thebibliography}{1}

\bibitem{BHJW}
Michael~D. Barrus, Stephen~G. Hartke, Kyle~F. Jao, and Douglas~B. West,
  \emph{Length thresholds for graphic lists given fixed largest and smallest
  entries and bounded gaps}, Discrete Math. \textbf{312} (2012), no.~9,
  1494--1501.

\bibitem{CSZZ}
Grant Cairns and Stacey Mendan, \emph{An improvement of a result of
  {Z}verovich--{Z}verovich}, Ars Math. Contemp., to appear.

\bibitem{HIS}
P.~L. Hammer, T.~Ibaraki, and B.~Simeone, \emph{Threshold sequences}, SIAM J.
  Algebraic Discrete Methods \textbf{2} (1981), no.~1, 39--49.

\bibitem{Li}
Shuo Yen~R. Li, \emph{Graphic sequences with unique realization}, J.
  Combinatorial Theory Ser. B \textbf{19} (1975), no.~1, 42--68.

\bibitem{TVW}
Amitabha Tripathi, Sushmita Venugopalan, and Douglas~B. West, \emph{A short
  constructive proof of the {E}rd{\H o}s-{G}allai characterization of graphic
  lists}, Discrete Math. \textbf{310} (2010), no.~4, 843--844.

\bibitem{TV}
Amitabha Tripathi and Sujith Vijay, \emph{A note on a theorem of {E}rd{\H o}s
  {$\&$} {G}allai}, Discrete Math. \textbf{265} (2003), no.~1-3, 417--420.

\bibitem{YL}
Jian-Hua Yin and Jiong-Sheng Li, \emph{Two sufficient conditions for a graphic
  sequence to have a realization with prescribed clique size}, Discrete Math.
  \textbf{301} (2005), no.~2-3, 218--227.

\bibitem{ZZ}
I.~{\`E}. Zverovich and V.~{\`E}. Zverovich, \emph{Contributions to the theory
  of graphic sequences}, Discrete Math. \textbf{105} (1992), no.~1-3, 293--303.

\end{thebibliography}
  \end{document}